\documentclass[onefignum,onetabnum]{siamart171218}

\usepackage{amsmath,amssymb,mathtools}
\usepackage{bm}
\usepackage{graphicx}
\usepackage{booktabs}
\usepackage{float}
\usepackage{algorithm}
\usepackage{algpseudocode}
\usepackage[numbers,sort&compress]{natbib}
\usepackage{needspace}
\usepackage{hyperref}

\newsiamthm{example}{Example}
\newsiamremark{remark}{Remark}

\newcommand{\tens}[1]{\mathcal{#1}}             
\newcommand{\starM}{\star_{\mathbf{M}}}         
\newcommand{\tp}{*}                             
\newcommand{\Mmat}{\mathbf{M}}                  
\newcommand{\hatT}[1]{\widehat{\mathcal{#1}}}   
\newcommand{\face}[2]{{#1}^{(#2)}}              
\newcommand{\lat}[2]{\vec{\mathcal{#1}}_{#2}}   
\newcommand{\Id}{\tens{I}}                      
\newcommand{\Rt}{\tens{R}}                      
\newcommand{\trans}{^{\mathsf{T}}}              
\newcommand{\fnorm}[1]{\left\lVert #1 \right\rVert_F}
\newcommand{\R}{\mathbb{R}}

\DeclareMathOperator{\rank}{rank}
\DeclareMathOperator{\svd}{svd}

\begin{document}

\title{A Symmetry-Preserving Tensor $\star_{\rm M}$-SVD}
\author{Victor Arsenescu\thanks{Department of Mathematics, Tufts University, Medford, MA 02155}
\and
Misha E.~Kilmer\thanks{Department of Mathematics, Tufts University,
    Medford, MA 02155 (\texttt{misha.kilmer@tufts.edu}).}
}
\date{\today}
\maketitle

\begin{abstract}
Multiway data such as image collections and video is ubiquitous,
but the usual approach of flattening them into matrices discards
the cross-mode structure that often carries the signal.
The t-product and its generalization, the $\starM$-product, give a
matrix-mimetic tensor algebra with a tensor SVD whose truncation is optimal in
the Frobenius norm, just as in the matrix case. Much real data also has
internal reflective symmetry: frontal faces, manufactured parts, and
leaves are all bilaterally symmetric. We define a symmetry-preserving $\starM$-SVD
that extends the matrix symmetry-preserving SVD of Shah and Sorensen
to the $\starM$-algebra. When a tensor's transform-domain frontal slices are
reflectively symmetric, the left basis of its $\starM$-SVD is symmetric too, so
only its top half must be stored. Bilaterally symmetric images, turned on their
side and stored as lateral slices, give such a tensor. This construction keeps
the symmetric part of each image and stores only half the basis.
We then recognize new images by projecting them onto this basis and matching to the
nearest training image in coefficient space. Across ten datasets of 
faces, leaves, butterfly wings, and other objects, the symmetric basis matches the
recognition rate of ordinary tensor SVD at $2$--$11$x less basis storage,
with the exception of MUCT. On faces under varying illumination it
exceeds the best rate the ordinary tensor SVD attains at any storage level.
\end{abstract} 

\begin{keywords}
multilinear algebra, tensor SVD, t-product, symmetry-preserving SVD,
facial recognition, data compression
\end{keywords} 

\section{Introduction}\label{sec:intro}

Data indexed by three or more modes, such as collections of images,
hyperspectral cubes, and time-resolved measurements, is common in imaging,
machine learning, and scientific computing.  A classical approach to compression and analysis on this data is to first matricize the data, then compute a rank-revealing factorization, such as the SVD. However, unfolding breaks the correlations
between modes in the multiway data set which is often what needs to be preserved.  
It has been shown that by keeping data in its native multiway format (i.e. as a tensor object) and decomposing the tensor directly yields better compression, better interpretability, and ease in downstream tasks (such as recognition).

Two of the most common tensor decomposition methods are the CANDECOMP/ PARAFAC
(CP) \cite{CarrollChang1970,Harshman1970} and Tucker models \cite{Tucker1966}, where the latter is most frequently used in the form of the higher-order SVD (HOSVD) \cite{DeLathauwer2000}.
The practical difficulty with a CP decomposition is that finding the exact tensor rank is a computationally intractable problem.
In fact, a best low-CP-rank approximation need not even exist \cite{deSilvaLim2008}.   
While the HOSVD can be computed in polynomial time, one needs to truncate the factor matrices to achieve a compressed representation.  However, 
truncating an HOSVD is quasi-optimal but not optimal in the Frobenius norm.
A newer family of tensor decompositions which can 
provide optimal low-rank approximations in the Frobenius norm does exist. In \cite{KilmerMartin2011}, Kilmer and Martin introduced the
t-product, a convolution-based tensor--tensor product under which third-order
tensors form a matrix-mimetic algebra with an SVD, the t-SVD. 
A few years later, Kernfeld, et al., \cite{KernfeldKilmerAeron2015}
 generalized the t-product to the $\starM$-product, in which the
defining convolution is replaced by any invertible linear transform $\Mmat$
applied along the third mode. Kilmer, et al., 
were able to prove that when $\Mmat$ is orthogonal/unitary (up to scale), truncating
the $\starM$-SVD gives the best low-rank approximation in the Frobenius
norm \cite{KHAN2021}.  Additionally, they were able to prove that the truncated tensor-SVD provides superior approximation properties to using a matricized-SVD to perform the compression.

A structural feature that the $\starM$ product family of tensor decompositions has not yet exploited is symmetry in the
data.  Many multiway datasets are reflectively symmetric: a frontal face is
bilaterally symmetric across its vertical midline, and a manufactured part
is symmetric by design.  In the matrix setting, Shah and
Sorensen showed that a data matrix whose rows are reflectively symmetric has an
SVD whose left singular vectors are themselves symmetric, so that only half of
each need be stored, and that for imperfectly symmetric data this construction
yields the Frobenius-norm-optimal symmetric approximation \cite{ShahSorensen2006}.
Despite the well-developed $\starM$-SVD, there is no dedicated treatment of the
symmetry-preserving case.

\paragraph{Contributions} We make the following contributions.
\begin{itemize}
  \item We show how a collection of bilaterally symmetric
    images, arranged as lateral slices turned on their side, yields a tensor
    whose transform-domain frontal slices are reflectively symmetric over their 
    middle rows. We refer to such a tensor as $\starM$-symmetric and
    show it is equivalent to face-wise matrix symmetry because the reflection
    commutes with the mode-3 transform.
  \item We establish the symmetry-preserving $\starM$-SVD
    $\tens{A}=\tens{U}\starM\tens{S}\starM\tens{V}\trans$ in which the left basis
    $\tens{U}=\tfrac{1}{\sqrt2}[\tens{U}_0;\,\Rt\starM\tens{U}_0]$ is exactly
    reflection-symmetric, so only its top half $\tens{U}_0$ need be stored and
    is computed from a half-size folded tensor. Since this factorization is
    itself a $\starM$-SVD, its truncation is likewise the best low-rank
    approximation in the Frobenius norm \cite{KHAN2021}, at half the basis
    storage.
  \item  We give an algorithm that halves the SVD size and basis storage
  and show on ten recognition datasets that it matches the ordinary tensor
  SVD's rate at $2$--$11\times$ less storage, except on MUCT, and on faces
  under varying illumination exceeds its best rate at any storage.
\end{itemize}

\paragraph{Organization}
Section~\ref{sec:background} recaps the standard background of the $\starM$-product.
Section~\ref{sec:symmetric} recalls matrix symmetry-preserving SVD and defines
$\starM$-symmetric tensors.  Section~\ref{sec:sym-svd} establishes the
symmetry-preserving $\starM$-SVD and its structure.
Section~\ref{sec:algorithms} presents the algorithm and its cost.
Section~\ref{sec:experiments} reports the recognition experiments, and
Section~\ref{sec:conclusion} concludes with future directions.

\section{Background}\label{sec:background}

We define the notation and standard results on the $\starM$-product and
$\starM$-SVD used throughout, following the presentation from \cite{KHAN2021},
to which we refer for proofs. The data in this paper are strictly real, so we define
everything for real tensors, but the complex versions hold with
conjugate transposes instead of transposes.

\paragraph{Notation}
A third-order tensor is an array $\tens{A}\in\R^{m\times p\times n}$, scalars
are lowercase ($a$), matrices capitals ($A$, $\Mmat$), tensors calligraphic
($\tens{A}$), using MATLAB notation for sub-arrays.  The three dimensions of a third-order tensor are called the {\it modes}.  The \emph{tube fibers} (tubes)
$\tens{A}_{i,j,:}$ are $1\times1\times n$,
and these are essentially vectors in the third mode. The
the $i$th \emph{frontal slice} is the
$m\times p$ matrix $\tens{A}_{:,:,i}$, written $\face{A}{i}$, the $j$th \emph{lateral slice}
is the $m\times1\times n$ tensor $\tens{A}_{:,j,:}$, written $\lat{A}{j}$, a ``column''
of $\tens{A}$ whose ``entries'' are tube fibers.  In this paper, the lateral slices carry the data
samples: each image is stored as one $\lat{A}{j}$ (turned on its side, see
Section~\ref{sec:tensor-construction}), so $\tens{A}$ holds $p$ images.
The Frobenius norm is $\fnorm{\tens{A}}^2=\sum_{i,j,k}\tens{A}_{i,j,k}^2$.
For an $n\times n$ matrix $\Mmat$, the mode-3 product $\tens{A}\times_3\Mmat$
applies $\Mmat$ to every tube of $\tens{A}$. It is mathematically equivalent to computing the matrix-matrix product $\Mmat\tens{A}_{(3)}$, 
where $\tens{A}_{(3)}$ denotes the mode-3 unfolding $\tens{A}_{(3)}\in\R^{n\times mp}$, and then reshaping the resulting product back
into a tensor (see 
\cite{KoldaBader2009} for details).

\paragraph{The $\starM$-product}
Let $\Mmat$ be an invertible $n\times n$ matrix.  We use hat notation for the
\emph{transform domain}: $\hatT{A}:=\tens{A}\times_3\Mmat$, so
$\tens{A}=\hatT{A}\times_3\Mmat^{-1}$, and $\face{\hat A}{i}=\hatT{A}_{:,:,i}$.
\begin{definition}[$\starM$-product \cite{KernfeldKilmerAeron2015}]\label{def:starM}
For $\tens{A}\in\R^{m\times p\times n}$ and $\tens{B}\in\R^{p\times\ell\times n}$,
$\tens{C}=\tens{A}\starM\tens{B}\in\R^{m\times\ell\times n}$ is defined by the
$n$ independent matrix products $\face{\hat C}{i}=\face{\hat A}{i}\face{\hat B}{i}$,
$i=1,\dots,n$, followed by $\tens{C}=\hatT{C}\times_3\Mmat^{-1}$.
\end{definition}
When $\Mmat$ is the (unnormalized) DFT matrix this is the t-product
$\tens{A}\tp\tens{B}$ from \cite{KilmerMartin2011}. Other choices in
\cite{KHAN2021} include the DCT, orthogonal wavelets, and data-driven
orthogonal $\Mmat$.  In  \cite{NewmanKeegan2025}, the authors discuss how to learn a best $\Mmat$. Throughout
this paper $\Mmat$ is the orthogonal DCT matrix, so $\Mmat^{-1}=\Mmat\trans$
and everything stays real.  The remaining algebra is defined facewise in the
transform domain \cite[Defs.~2.1--2.3]{KHAN2021}: the \emph{transpose}
$\tens{A}\trans\in\R^{p\times m\times n}$ has faces
$(\face{\hat A}{i})\trans$ (so
$(\tens{A}\starM\tens{B})\trans=\tens{B}\trans\starM\tens{A}\trans$). The
\emph{identity} $\Id\in\R^{m\times m\times n}$ has every $\face{\hat I}{i}$
equal to the $m\times m$ identity, while $\tens{Q}\in\R^{m\times m\times n}$ is
\emph{$\starM$-orthogonal} if $\tens{Q}\trans\starM\tens{Q}=\Id=\tens{Q}\starM\tens{Q}\trans$,
i.e.\ every $\face{\hat Q}{i}$ is orthogonal, and
$\tens{U}\in\R^{m\times r\times n}$ has \emph{$\starM$-orthonormal columns}
(lateral slices) if $\tens{U}\trans\starM\tens{U}$ is the $r\times r\times n$
identity.

\subsection{Truncating the \texorpdfstring{$\starM$}{starM}-SVD}
Following \cite{KHAN2021} we restrict to $\Mmat=cW$ with $W$ orthogonal and
$c\neq0$ (for the DCT, $c=1$) then $\fnorm{\tens{Q}\starM\tens{B}}=\fnorm{\tens{B}}$
for $\starM$-orthogonal $\tens{Q}$ \cite[Thm.~3.1]{KHAN2021}, and Frobenius
norms may be measured face-wise in the transform domain,
$\fnorm{\tens{B}}^2=\tfrac1{c^2}\sum_i\fnorm{\face{\hat B}{i}}^2$.

\begin{definition}[$\starM$-SVD \cite{KilmerMartin2011,KernfeldKilmerAeron2015}, {\cite[Def.~3.2]{KHAN2021}}]\label{def:tsvdm}
The $\starM$-SVD (t-SVDM) of $\tens{A}\in\R^{m\times p\times n}$ is
\begin{equation}\label{eq:tsvdm}
  \tens{A}=\tens{U}\starM\tens{S}\starM\tens{V}\trans
   =\sum_{i=1}^{r}\tens{U}_{:,i,:}\starM\tens{S}_{i,i,:}\starM\tens{V}_{:,i,:}\trans ,
\end{equation}
with $\tens{U}$, $\tens{V}$ $\starM$-orthogonal, $\tens{S}$ f-diagonal (every
frontal slice diagonal), and \\ $r\le\min(m,p)$ the number of nonzero
\emph{singular tubes} $\mathbf{s}_i=\tens{S}_{i,i,:}$.
\end{definition}

The factorization is computed face-wise: form $\hatT{A}$, take the matrix SVD
$\face{\hat A}{i}=\face{\hat U}{i}\face{\hat S}{i}(\face{\hat V}{i})\trans$
of each face, and transform the factors back with $\times_3\Mmat^{-1}$
\cite[Alg.~2]{KHAN2021}, with the economy version keeping $\min(m,p)$ columns.
Then $\fnorm{\tens{A}}^2=\sum_i\fnorm{\mathbf{s}_i}^2$ with
$\fnorm{\mathbf{s}_1}\ge\fnorm{\mathbf{s}_2}\ge\cdots$
\cite[Cor.~3.3]{KHAN2021}. The \emph{t-rank} of $\tens{A}$ is $r$ (not to be
confused with the CP rank), the \emph{multirank} is the vector $\rho$ with
$\rho_i=\rank\face{\hat A}{i}$, and the \emph{implicit rank} is $\sum_i\rho_i$,
the total number of singular values kept across the faces
\cite[Defs.~3.4--3.6]{KHAN2021}.

\begin{theorem}[{Eckart-Young for the $\starM$-SVD \cite[Thm.~3.7]{KHAN2021}}]\label{thm:EY}
Let $\Mmat=cW$ as above and
$\tens{A}_k=\tens{U}_{:,1:k,:}\starM\tens{S}_{1:k,1:k,:}\starM\tens{V}_{:,1:k,:}\trans$.
Then $\tens{A}_k$ is the best Frobenius-norm approximation to $\tens{A}$
among all tensors $\tens{X}\starM\tens{Y}$ with
$\tens{X}\in\R^{m\times k\times n}$, $\tens{Y}\in\R^{k\times p\times n}$
(the t-rank-$k$ tensors), and
$\fnorm{\tens{A}-\tens{A}_k}^2=\sum_{i>k}\fnorm{\mathbf{s}_i}^2$.
\end{theorem}
\begin{theorem}[{Multirank truncation \cite[Thm.~3.8]{KHAN2021}}]\label{thm:multirank}
Let $\tens{A}_\rho$ keep the leading $\rho_i$ singular values and vectors of face $i$,
$(\hatT{A}_\rho)_{:,:,i}=\hat U_{:,1:\rho_i,i}\hat S_{1:\rho_i,1:\rho_i,i}\hat V_{:,1:\rho_i,i}\trans$.
Then $\tens{A}_\rho$ is the best multirank-$\rho$ approximation to $\tens{A}$
in the Frobenius norm and
$\fnorm{\tens{A}-\tens{A}_\rho}^2=\tfrac1{c^2}\sum_{i=1}^n\sum_{j>\rho_i}(\hat\sigma^{(i)}_j)^2$,
where $\hat\sigma^{(i)}_j$ are the singular values of $\face{\hat A}{i}$.
\end{theorem}
Theorem~\ref{thm:EY} truncates every face to the same $k$, while the t-SVDMII of
\cite{HaoKBH2013,KHAN2021} instead chooses $\rho$ from the data: sort all
$(\hat\sigma^{(i)}_j)^2$ in decreasing order, keep the largest ones whose
cumulative share of $\fnorm{\tens{A}}^2$ is at most $\gamma$ (the variant in
\cite[Alg.~3]{KHAN2021} keeps one further value, the one that first exceeds
$\gamma$), and keep in each face the singular values (with their singular
vectors) whose squared singular value is at least the last kept one.
Algorithm~\ref{alg:symsvd} below is this truncation applied to a $\starM$-symmetric tensor.

Finally, the tensor truncation is never worse than the matrix one on the same
data.
\begin{theorem}[{\cite[Thms.~5.3, 5.5]{KHAN2021}}]\label{thm:tensor-beats-matrix}
Stack the same $p$ images as columns of $A\in\R^{mn\times p}$ and let $A_k$ be
its best rank-$k$ approximation.  Then
$\fnorm{\tens{A}-\tens{A}_k}\le\fnorm{A-A_k}$, and there is a $\gamma$ for
which the t-SVDMII approximation $\tens{A}_\rho$ has implicit rank at most
that of $\tens{A}_k$ and $\fnorm{\tens{A}-\tens{A}_\rho}\le\fnorm{\tens{A}-\tens{A}_k}$.
\end{theorem}
The first inequality holds because the truncated matrix SVD is itself a
t-rank-$k$ tensor, so Theorem~\ref{thm:EY} applies. The tensor basis
$\tens{U}_{:,1:k,:}$ and the matrix basis $U_{:,1:k}$ both require the storage of $kmn$ numbers.

\section{\texorpdfstring{$\starM$}{starM}-Symmetric Tensors}\label{sec:symmetric}
As noted in the previous section, the $\starM$-SVD computation relies on matrix-based SVDs in the so-called transform domain.  
In order to devise a symmetry-preserving $\starM$-SVD, we therefore first describe the symmetry-preserving matrix-based SVD and we will use this to build our tensor decomposition.  

\subsection{The matrix symmetry-preserving SVD}\label{sec:spsvd}
We recall the matrix construction we extend.  Let $B\in\R^{m\times p}$ ($m=2h$) have
reflectively symmetric rows: with $X_0=B_{1:h,:}$, $X_1=B_{h+1:m,:}$ and the
$h\times h$ row-reversal $R$ (the anti-identity, $R\trans=R$, $R^2=I$), one has
$RX_1=X_0$.  Shah and Sorensen \cite{ShahSorensen2006} observed that the (economy)
SVD of $B$ is obtained from the half-size matrix $\bar B=\tfrac12(X_0+RX_1)$: if
$\bar B=U_0S_0V_0\trans$, then
\begin{equation}\label{eq:matspsvd}
  U=\tfrac{1}{\sqrt2}\begin{bmatrix}U_0\\ RU_0\end{bmatrix},\qquad
  S=\sqrt2\,S_0,\qquad V=V_0,
\end{equation}
gives $B=USV\trans$ with $U$ exactly symmetric ($U_{h+1:m,:}=RU_{1:h,:}$).  When
$RX_1=X_0+E$ (imperfect symmetry), the same $U,S,V$ built from
$\bar B=\tfrac12(X_0+RX_1)$ yield the Frobenius-norm-optimal symmetric
approximation of $B$ \cite[Thm.~5.1]{ShahSorensen2006}. Equivalently, every $X_0$ 
splits into a symmetric part $\tfrac12(X_0+RX_1)$ and an antisymmetric part
$\tfrac12(X_0-RX_1)$, and the construction keeps only the former.

\begin{example}[A row-symmetric $4\times4$ matrix]\label{ex:44}
Let $R=\left[\begin{smallmatrix}0&1\\1&0\end{smallmatrix}\right]$ and
\[
  B=\begin{bmatrix}4&3&0&0\\ 0&0&2&1\\ 0&0&2&1\\ 4&3&0&0\end{bmatrix},
  \qquad X_0=\begin{bmatrix}4&3&0&0\\0&0&2&1\end{bmatrix},\quad
  X_1=\begin{bmatrix}0&0&2&1\\4&3&0&0\end{bmatrix},
\]
so $RX_1=X_0$.  The fold $\bar B=\tfrac12(X_0+RX_1)=X_0$ has singular values
$s_0=(5,\sqrt5)$, and \eqref{eq:matspsvd} $B$ has singular values
$\sqrt2\,s_0=(\sqrt{50},\sqrt{10})$. The best rank-one approximation has relative
error $\sqrt{10}/\sqrt{60}\approx0.408$.  Because the left singular vectors are
symmetric, only their top halves need be stored.
\end{example}

\subsection{Extending the construction to tensors}
We now extend the reflective structure of Section~\ref{sec:spsvd} to the
$\starM$-algebra.  Let $\Rt\in\R^{h\times h\times n}$ be the reflection
tensor whose every transform-domain frontal slice equals the $h\times h$
row-reversal $R$, i.e.\ $\face{\hat R}{i}=R$ for all $i$.  Then $\Rt\starM\tens{B}$
has transform-domain frontal slices $R\,\face{\hat B}{i}$, so $\Rt\starM$ acts as
the facewise row-reflection. Row reflection acts on mode~1 and the transform on mode~3,
so they commute: $\Rt\starM\tens{B}$ is exactly $\tens{B}$ with its rows reflected in the spatial domain.

\begin{definition}[$\starM$-symmetric tensor] \label{def:sym}
A tensor $\tens{A}\in\R^{m\times p\times n}$ ($m=2h$) is
$\starM$-symmetric if, writing
$\tens{X}_0=\tens{A}_{1:h,:,:}$ and $\tens{X}_1=\tens{A}_{h+1:m,:,:}$, one has
\[
  \tens{X}_0=\Rt\starM\tens{X}_1,
\]
equivalently, in the transform domain, $\face{\hat A}{i}_{1:h,:}
=R\,\face{\hat A}{i}_{h+1:m,:}$ for every $i$: each frontal slice
$\face{\hat A}{i}$ is row-symmetric.
\end{definition}

\begin{lemma}[Facewise characterization]\label{lem:facewise}
$\tens{A}$ is $\starM$-symmetric if and only if every transform-domain
frontal slice $\face{\hat A}{i}$ is row-symmetric in the sense of
Section~\ref{sec:spsvd}.
\end{lemma}
\begin{proof}
Apply $\Mmat$ along the third mode.  Since $R$ is constant across that mode, the
tensor relation $\tens{X}_0=\Rt\starM\tens{X}_1$ is, facewise,
$\face{\hat X_0}{i}=R\,\face{\hat X_1}{i}$ for every $i$ (the reflection commutes
with the mode-3 transform).  The result follows because $\Mmat$ is invertible.
The commutation is exact: $\Rt$ acts on mode~1 and $\Mmat$ on mode~3, so the two
operations commute for any invertible $\Mmat$.
\end{proof}

\subsection{From bilaterally symmetric images to a row-symmetric tensor}
\label{sec:tensor-construction}
Definition~\ref{def:sym} is a statement about the \emph{rows} of the frontal
slices, whereas the symmetry we want to exploit in an image collection is
\emph{bilateral}: a frontal face, a leaf, or a butterfly is (approximately)
symmetric across a vertical axis, i.e.\ across the columns of the image.  The two
are reconciled by how the images are placed in the tensor, so we state this
explicitly.

Let $I_1,\dots,I_p\in\R^{n\times m}$ be a collection of $p$ images, each with $n$
rows and $m$ columns of pixels, $m=2h$, whose bilateral axis is the vertical
midline, so that ideally $I_j(\ell,i)=I_j(\ell,m+1-i)$. Following
\cite{HaoKBH2013}, each image is stored as a \emph{lateral} slice of
$\tens{A}\in\R^{m\times p\times n}$. Before it is inserted, the image is
turned on its side (a $90^\circ$ rotation) and
\begin{equation}\label{eq:tensor-construction}
  \tens{A}(i,j,\ell)=I_j\bigl(n+1-\ell,\;i\bigr),\qquad
  i=1,\dots,m,\quad j=1,\dots,p,\quad \ell=1,\dots,n .
\end{equation}
Thus mode~1 of $\tens{A}$ indexes the \emph{horizontal} position within an
image, mode~2 indexes the images, and mode~3 (the tube direction, along which
$\Mmat$ acts) indexes the vertical position.  In words: the $\ell$-th
frontal slice $\face{A}{\ell}\in\R^{m\times p}$ holds one row of every
image as a column, so its rows are horizontal positions across the object.  A
bilaterally symmetric image therefore contributes a column that is symmetric
under the row-reversal $R$, and if every image is bilaterally symmetric then
every frontal slice is row-symmetric and $\tens{A}$ is
$\starM$-symmetric in the sense of Definition~\ref{def:sym} (by
Lemma~\ref{lem:facewise}, the mode-3 transform preserves this).  With
this placement the reflection tensor $\Rt$ acts on mode~1, i.e.\ it mirrors each
image about its vertical axis, and the ``top half'' $\tens{U}_0$ of the basis
in Theorem~\ref{thm:symsvd} is precisely the left half of each basis image.

\begin{remark}
One could instead transpose the images and place them
as \emph{frontal} slices, which also makes each face row-symmetric. 
We keep the lateral-slice convention from \cite{HaoKBH2013,KBHH2013}: it makes
each image a ``vector'' $\lat{A}{j}$ in the $\starM$-module, so that the
$\starM$-SVD compresses across the collection in the same way the matrix
SVD compresses a data matrix whose columns are images, and the tube transform
$\Mmat$ acts along a spatial mode of the image, where it can exploit correlations
between neighboring rows.
\end{remark}

\section{The Symmetry-Preserving \texorpdfstring{$\starM$}{starM}-SVD} \label{sec:sym-svd} 
We now state the main factorization, which we call the symmetry-preserving
tensor SVD (sptSVD).

\begin{theorem}[Symmetry-preserving $\starM$-SVD] \label{thm:symsvd}
Let $\tens{A}\in\R^{m\times p\times n}$ ($m=2h$) be $\starM$-symmetric
with $\tens{X}_0=\tens{A}_{1:h,:,:}$, $\tens{X}_1=\tens{A}_{h+1:m,:,:}$, and let
\[
  \tens{B}=\tfrac12\bigl(\tens{X}_0+\Rt\starM\tens{X}_1\bigr)\in\R^{h\times p\times n}
\]
have $\starM$-SVD $\tens{B}=\tens{U}_0\starM\tens{S}_0\starM\tens{V}_0\trans$.
Then
\[
  \tens{A}=\tens{U}\starM\tens{S}\starM\tens{V}\trans,\qquad
  \tens{U}=\tfrac{1}{\sqrt2}\begin{bmatrix}\tens{U}_0\\[2pt]\Rt\starM\tens{U}_0\end{bmatrix},
  \quad \tens{S}=\sqrt2\,\tens{S}_0,\quad \tens{V}=\tens{V}_0,
\]
is an economy $\starM$-SVD of $\tens{A}$ in which $\tens{U}\in\R^{m\times r\times n}$
($r\le h$) has $\starM$-orthonormal columns ($\tens{U}\trans\starM\tens{U}$ is the
$r\times r\times n$ identity) and is itself $\starM$-symmetric,
$\tens{U}_{h+1:m,:,:}=\Rt\starM\tens{U}_{1:h,:,:}$.
Consequently only the top half $\tens{U}_0$ (and $\tens{V}$) need be stored.  The
singular tubes of $\tens{S}$ satisfy $\fnorm{s_j}^2=2\fnorm{(s_0)_j}^2$ and are
ordered by non-increasing Frobenius norm.
\end{theorem}
\begin{proof}
Apply $\Mmat$ along the third mode.  By Lemma~\ref{lem:facewise} each
$\face{\hat A}{i}$ is row-symmetric, so the matrix construction
\eqref{eq:matspsvd} of Shah and Sorensen applies face-wise: with
$\face{\hat B}{i}=\tfrac12(\face{\hat X_0}{i}+R\face{\hat X_1}{i})$ and its SVD
$\face{\hat U_0}{i}\face{\hat S_0}{i}(\face{\hat V_0}{i})\trans$, the matrices
$\face{\hat U}{i}=\tfrac1{\sqrt2}[\face{\hat U_0}{i};\,R\face{\hat U_0}{i}]$,
$\face{\hat S}{i}=\sqrt2\,\face{\hat S_0}{i}$, $\face{\hat V}{i}=\face{\hat V_0}{i}$
form an SVD of $\face{\hat A}{i}$ with $\face{\hat U}{i}$ symmetric and orthonormal.
Collect these face-wise factors and invert $\Mmat$; the result follows because
$\Mmat$ is invertible and $\Rt\starM$ acts face-wise as $R$. Each face-wise SVD
lists its singular values in non-increasing order, so the $j$th singular tube,
whose squared Frobenius norm is $\sum_i(\face{\hat S}{i})_{jj}^2$ (up to the
constant of $\Mmat$), has a non-increasing norm in $j$ as well.
\end{proof}

A major consequence of Theorem~\ref{thm:symsvd} is an Eckart--Young theorem
for the sptSVD: because $\tens{U}\starM\tens{S}\starM\tens{V}\trans$ is
itself a $\starM$-SVD, Theorem~\ref{thm:EY} applies, so truncating to the
$k$ largest singular tubes gives the best $\starM$-rank-$k$ approximation of
$\tens{A}$ in the Frobenius norm, and it is still symmetric.

\section{Algorithms}\label{sec:algorithms}
\begin{algorithm}[t]
\caption{Symmetry-preserving $\starM$-SVD (sptSVD)}\label{alg:symsvd}
\begin{algorithmic}[1]
\Require $\starM$-symmetric tensor $\tens{A}\in\R^{m\times p\times n}$ ($m=2h$),
  orthogonal transform $\Mmat$ (we use the DCT), energy fraction $\gamma\in[0,1]$.
\State $\hatT{A}\gets\tens{A}\times_3\Mmat$ \Comment{move to the transform domain}
\For{each transform slice $\hat A_i$, $i=1,\dots,n$}
  \State split $\hat A_i$ into its top and bottom halves $\hat A_i^{\text{top}},\hat A_i^{\text{bot}}$ (each $h\times p$)
  \State $F_i\gets\tfrac12\bigl(\hat A_i^{\text{top}}+R\,\hat A_i^{\text{bot}}\bigr)$
    \Comment{fold: average the two halves ($R$ flips rows)}
  \State $[U_i,\sigma_i,V_i]\gets\svd(F_i)$ \Comment{one economy SVD, half the rows}
  \State $e_i\gets 2\,\sigma_i^2$ \Comment{energy of each symmetric component}
\EndFor
\State pool all the $e_i$, sort them large-to-small, and keep the largest ones whose
  cumulative share of the total energy is at most $\gamma$ (at least one). Let $\tau$ be the smallest kept value
\For{each slice $i=1,\dots,n$}
  \State keep the components with $e_i\ge\tau$. for those, set
  \Statex \hspace{1.4em}top half of basis $\gets U_i/\sqrt2$,\quad
    singular values $\gets\sqrt2\,\sigma_i$,\quad right factor $\gets V_i$
\EndFor
\State transform the kept factors back with $\times_3\Mmat\trans$ to get
  $\tens{U}_0$, $\tens{S}$, $\tens{V}$
\Ensure top half $\tens{U}_0$, singular values $\tens{S}$, and $\tens{V}$.
  The bottom half of the basis is $R\,\tens{U}_0$ and is not stored. In
  practice the factors are kept and used in the transform domain
  (Section~\ref{sec:experiments}). The final inverse transform
  $\times_3\Mmat\trans$ is only meant to express the factorization in the spatial domain.
\end{algorithmic}
\end{algorithm}

\paragraph{Truncation strategy}
Algorithm~\ref{alg:symsvd} does not keep the same number of components in every face. It ranks the symmetric-component energies from all $n$ faces together against one global threshold, and each face keeps only the components of its own that clear it, so a face carrying more energy keeps more components than one carrying less, and the per-face count $\rho_i$ varies.

\begin{corollary}[Comparison with the matrix symmetry-preserving SVD]
\label{cor:beats-matrix}
Let $\tens{A}\in\R^{m\times p\times n}$ be $\starM$-symmetric with
$\Mmat$ orthogonal, and let $A\in\R^{mn\times p}$ be the same data
flattened, one image per column.  Let $A^{\mathrm{sym}}_k$ be the rank-$k$
symmetry-preserving matrix approximation of $A$ of Shah and Sorensen
\cite{ShahSorensen2006}, and $\tens{A}_k$ the rank-$k$ sptSVD truncation of
Theorem~\ref{thm:symsvd}.  Then
\[
  \fnorm{\tens{A}-\tens{A}_k}\;\le\;\fnorm{A-A^{\mathrm{sym}}_k},
\]
and both sides store the same number of basis entries, since each keeps only
the top half of its basis.  

Moreover, if $\tens{A}_k$ keeps $r$ singular values in all and $\gamma$ is the share of
$\fnorm{\tens{A}}^2$ carried by the $r$ globally largest squared singular values of $\tens{A}$ (so
$\gamma\ge\fnorm{\tens{A}_k}^2/\fnorm{\tens{A}}^2$) and the $r$th and $(r{+}1)$st
largest squared singular values are not tied, then run with this $\gamma$ matches or
beats $\tens{A}_k$ in error while keeping no more components, which is the analogue
of Theorem~5.5 from \cite{KHAN2021}.
\end{corollary}

\begin{proof}
Theorem~\ref{thm:symsvd} showed the sptSVD is a genuine $\starM$-SVD of
$\tens{A}$, so $\tens{A}_k$ is the best t-rank-$k$ approximation of
$\tens{A}$ there is \cite[Thm.~3.7]{KHAN2021}. Theorem~\ref{thm:tensor-beats-matrix}
says the best t-rank-$k$ tensor approximation is at least as
good as the best rank-$k$ matrix approximation of the flattened data, and
$A^{\mathrm{sym}}_k$ is a rank-$k$ matrix, so it cannot beat the best one.
For the storage claim: the sptSVD keeps $\tens{U}_0$, which holds $h\times k\times n$ numbers,
and the matrix construction keeps the top halves of $k$ vectors of length $mn$, which is
also $hkn$ numbers.

For the last claim, the energies $e_i=2\sigma_i^2$ pooled by
Algorithm~\ref{alg:symsvd} are the squared singular values of $\tens{A}$
(Theorem~\ref{thm:symsvd}), and the squared error of any truncation equals
the energy it discards (Theorem~\ref{thm:multirank}).  With $\gamma$ as
chosen, the algorithm keeps exactly the $r$ globally largest values
(the no-tie assumption makes the cut unambiguous).  These capture at least
as much energy as the $r$ values $\tens{A}_k$ keeps, so the algorithm keeps
no more components and discards no more energy, hence no more error.
\end{proof}

\paragraph{Cost} Algorithm~\ref{alg:symsvd} transforms once along mode~3 and then
takes one economy SVD per frontal slice.  Each SVD is on the $h\times p$ fold, half
the row dimension of the ordinary $\starM$-SVD's $m\times p$ slice, so each one is
smaller.  For storage, both methods keep the same $\tens{V}$, so the only
difference is in $\tens{U}$: it drops from $md$ to $hd=\tfrac12 md$ numbers for $d$
retained tubes, so the basis storage halves.

\section{Numerical Experiments}\label{sec:experiments}
We evaluate the symmetry-preserving $\starM$-SVD on identity recognition from
pre-aligned image collections, as in \cite{HaoKBH2013}.  Symmetry raises
raw accuracy only modestly, so we instead measure how much basis storage
the symmetric basis needs to reach the accuracy of the ordinary tensor SVD.
We begin by describing the data preparation and the recognition protocol,
since every result that follows depends on them.

\subsection{Data and tensor construction}\label{sec:setup}
Every image is converted to grayscale (one channel per image,
as in \cite{HaoKBH2013}), resized to $64\times64$ pixels,
and scaled to $[0,1]$.  Sources that are not already square are
center-cropped first, except LFW, Extended Yale~B, and MUCT, which are given as
fixed non-square face crops and are resized directly (preserving
left--right symmetry). The leaf and butterfly images are instead cropped to
their bounding boxes, as described below.
We use every class (identity or species) for datasets with at most $100$ classes, and for the remaining three (CelebA, LeafSnap, MUCT), whose class counts run into the hundreds or thousands with only tens of images per class, we keep the $15$ most populous classes, ties broken by label order (Table~\ref{tab:datasets}). 
For the largest datasets we also cap the number of
images per class (between $40$ and $500$, depending on the dataset) to keep the
runtime manageable. For Extended Yale~B we use the
frontal pose and keep the $32$ mildest of its
$64$ illuminations per subject (ranked by
flash angle, with the most extreme almost entirely in shadow). The full $64$-illumination
set is evaluated separately (see ``all illuminations'' in Table~\ref{tab:main}).

\begin{table}[!htbp]
\centering
\begin{tabular}{lrr}
\toprule
dataset & \# classes & \# images \\
\midrule
Extended Yale~B ($32$ mildest illuminations) & 38 & 1216 \\
Extended Yale~B (all illuminations) & 38 & 2414 \\
AFHQ & 3 & 1500 \\
Olivetti & 40 & 400 \\
LFW & 62 & 3023 \\
CelebA & 15 & 451 \\
Swedish leaves & 15 & 1125 \\
Leeds butterflies & 10 & 832 \\
LeafSnap & 15 & 1200 \\
COIL-100 & 100 & 7200 \\
MUCT & 15 & 225 \\
\bottomrule
\end{tabular}
\caption{Classes and images used from each collection, after the selection
described in the text (for LFW, the people with $20$ images or more).
Every image is resized to $64\times64$, so $m=n=64$ and $h=32$ throughout.}
\label{tab:datasets}
\end{table}

The face datasets are pre-aligned, with their bilateral axis
at the vertical midline. The leaf and butterfly images we align ourselves. Using each
image's segmentation mask, the image is rotated so that the leaf's mid-vein or 
butterfly's body axis, respectively, is vertical, then cropped to the object's bounding box.
Three further sets (beetles, cars, BIOSCAN) were aligned the same way and evaluated, but they
are omitted from Table~\ref{tab:main} (see Section~\ref{sec:matched}).
With the bilateral axis at the vertical midline, the mode-1 reflection $R$ (the
$h\times h$ row-reversal, $h=m/2=32$) coincides with that axis, and a
frontal slice is symmetric across its middle row exactly when the underlying
images are bilaterally symmetric.  

Real images are only approximately symmetric.
We impose no exact-symmetry assumption: folding each slice with its reflection
gives the nearest symmetric tensor to the data
\cite[Thm.~5.1]{ShahSorensen2006}, so the basis is built from the data's best
symmetric approximation.

Images are placed on their sides in the tensor exactly as in Section~\ref{sec:tensor-construction}, so $\tens{A}\in\R^{m\times p\times n}$ with $m=n=64$. We take $\Mmat$ to be the orthonormal DCT (DCT-II)
along mode~3 throughout.

\subsection{Recognition protocol}\label{sec:protocol}
We use the recognition procedure of Hao, Kilmer, Braman, and Hoover
\cite[Alg.~5, ``T-SVD Method~II'']{HaoKBH2013}, the $\starM$ analogue of the
eigenfaces procedure (project onto the basis, nearest neighbor in coefficient
space), with the DCT in place of the DFT since our data is strictly real-valued. 
For one dataset and one random seed, the steps are as follows:

\begin{enumerate}
  \item \emph{Split.}  Within each class the images are shuffled and $70\%$
    go to training, the rest to test, so every class appears in both 
    (a class with a single image would be dropped).
    Let $\tens{A}\in\R^{m\times p\times n}$
    hold the $p$ training images as lateral slices and let $\lat{T}{}$
    denote a test image, arranged the same way.
  \item \emph{Centering.}  The mean lateral slice
    $\bar{\lat{A}{}}=\frac1p\sum_{j}\lat{A}{j}$ of the \emph{training} images is
    subtracted from every training and test image. Note that the test images play no
    part in the mean, the basis, or the choice of truncation.
  \item \emph{Basis.}  The centered training tensor is moved to the
    transform domain once and each frontal slice is factored there by one of
    three constructions, described below. Truncation is by the energy rule from Algorithm~\ref{alg:symsvd} with fraction $\gamma$.
    The retained vectors are kept in the transform domain and never
    transformed back. Here, $\gamma$ ranges over
\[\{0.5,\,0.6,\,0.7,\,0.8,\,0.85,\,0.9,\,0.925,\,0.95,\,0.97,\,0.98,\,0.99,\,0.995,\,0.999\};\]
    For the symmetric constructions, each face has at most $h=32$ vectors.
  \item \emph{Coefficients.}  Each centered training image and the centered
    test image are moved to the transform domain and projected face by face
    onto the retained vectors, one $\rho_i\times m$ by $m\times1$ product per
    face (for the symmetric constructions the full vectors are rebuilt from their stored halves), 
    and the coefficient vectors of all faces are concatenated. Since
    $\Mmat$ is orthogonal, distances between these coefficients equal the
    Frobenius distances $\fnorm{\tens{U}\trans\starM(\lat{A}{j}-\lat{A}{l})}$
    in the spatial domain.
  \item \emph{Match.} The test image is assigned the label of the training
    image whose coefficient tensor is nearest in the Frobenius norm,
    $j^\star=\arg\min_j\fnorm{\lat{C}{\mathrm{test}}-\lat{C}{j}}$ (nearest neighbor).
  \item \emph{Score.}  The \emph{recognition rate} is the fraction
    of test images whose assigned label is correct.  It is computed for every
    $\gamma$ and every construction, and reported as the mean over $5$ random
    splits (seeds $0$--$4$).
\end{enumerate}
The three basis constructions are:
\begin{itemize}
\item \textsf{plain}, the ordinary $\starM$-SVD basis (full $\tens{U}$, $m$ numbers
per retained vector), \item \textsf{new}, the symmetry-preserving basis of
Theorem~\ref{thm:symsvd} with $R$ the row-reversal (only the top half
$\tens{U}_0$, $h$ numbers per retained vector, is stored), and \item \textsf{rand},
a control that runs the same fold-and-halve construction with $R$ replaced by
a random pairing of the rows, i.e.\ a random symmetric permutation matrix with
no fixed points ($R=R\trans$, $R^2=I$, one draw per split, shared by all
faces and used for every $\gamma$), instead of the reflection.  Any
such $R$ is itself a reflection of the row index set, so \textsf{rand} is
Theorem~\ref{thm:symsvd} applied with the wrong mirror. 
\end{itemize}
The purpose of the \textsf{rand} is to act as a
control that separates
the effect of the mirror axis itself from the automatic $2\times$ storage cut that
\emph{any} row-pairing provides: \textsf{rand} also stores half a basis, but the
halves it pairs have nothing to do with the object's symmetry.
Folding averages each row with its partner, so
the mirror pairing barely changes a bilaterally symmetric image, while a random
pairing scrambles it (see Figure~\ref{fig:rand-example}). A blurred half-size image still carries much of the class
information, which is why \textsf{rand} typically falls between \textsf{plain}
and \textsf{new}.  A genuine symmetry effect is present only when
\textsf{new} $>$ \textsf{rand}.

\begin{figure}[!htbp]
\centering
\includegraphics[width=0.8\textwidth]{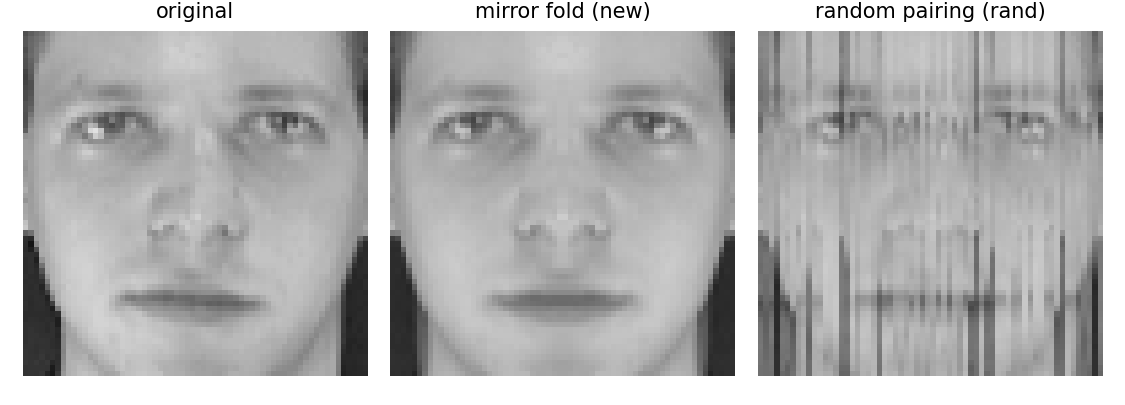}
\caption{One Olivetti face, the part kept by the mirror fold (\textsf{new}),
and the part kept by one random pairing (\textsf{rand}).  The random pairing
averages each column of the left half with an unrelated column of the right
half, so it halves the basis in the same way but discards the object's
symmetry.}
\label{fig:rand-example}
\end{figure}

\paragraph{Storage accounting}  We count the basis storage that the
recognition procedure keeps: the transform-domain vectors of all
faces, $m\sum_i\rho_i$ real numbers for \textsf{plain} and $h\sum_i\rho_i$
for \textsf{new} and \textsf{rand} (their own $\rho_i$'s).  The right
factor $\tens{V}$ and the singular tubes are not stored, since only the left
vectors are needed to form coefficients.  (Every construction also stores
the training mean slice, $mn$ numbers, which is identical across constructions
and omitted from the counts). The symmetric constructions store half as much
per vector but retain different numbers of them, so we compare at matched \emph{recognition rate}.

\subsection{Metrics}\label{sec:metrics}
The setup of Sections~\ref{sec:setup}--\ref{sec:protocol} is fixed throughout,
with $\Gamma$ denoting the $\gamma$ grid. For a construction
$a\in\{\textsf{plain},\textsf{new},\textsf{rand}\}$ and an energy fraction
$\gamma\in\Gamma$, let $\rho_i^{a}(\gamma)$ be the number of vectors kept in
face $i$ (Algorithm~\ref{alg:symsvd}) and $r_a(\gamma)$ the recognition rate,
both averaged over the five seeds.  Then:
\begin{enumerate}
  \item \emph{Storage} of construction $a$ at $\gamma$: the transform-domain entries
    kept, $S_a(\gamma)=m\sum_{i=1}^{n}\rho_i^{a}(\gamma)$ for \textsf{plain}
    and $S_a(\gamma)=h\sum_{i=1}^{n}\rho_i^{a}(\gamma)$ for \textsf{new} and
    \textsf{rand} (only the top half $\tens{U}_0$ is stored).
  \item \emph{Best recognition rate}: $r_a^\star=\max_{\gamma\in\Gamma}r_a(\gamma)$.
    The gaps in Table~\ref{tab:main} are
    $r_{\textsf{new}}^\star-r_{\textsf{plain}}^\star$ and
    $r_{\textsf{new}}^\star-r_{\textsf{rand}}^\star$.
  \item \emph{Rate bar}: $\beta=r_{\textsf{plain}}^\star-0.01$ (the offset
    keeps seed noise from deciding the matched point).
  \item \emph{Storage to reach the bar}:
    $S_a^{\beta}=\min\{S_a(\gamma):\gamma\in\Gamma,\ r_a(\gamma)\ge\beta\}$,
    undefined (``---'') if no $\gamma$ reaches it.
  \item \emph{Storage saving}: $\times\text{less}_a=S_{\textsf{plain}}^{\beta}/S_a^{\beta}$
    for $a\in\{\textsf{new},\textsf{rand}\}$.
  \item \emph{Margin}: the mean over $\Gamma$ of
    $r_{\textsf{new}}(\gamma)-r_{\textsf{rand}}(\gamma)$, used to color
    Figure~\ref{fig:main}.
\end{enumerate}
Both the bar and the matched points are read off the mean test curves
with no separate validation split for $\gamma$. The rule is the same
for every construction, but the absolute rates are optimistic (each is
a maximum over $\gamma$, averaged across five seeds).

\subsection{Storage at matched recognition rate}\label{sec:matched}
Table~\ref{tab:main} reports ten datasets, chosen to cover every outcome we
observed (Figure~\ref{fig:main} plots the storage saving). Five further sources
(BIOSCAN insects, beetles, cars, ALOI, the
original Yale faces) are omitted because every construction's rate was
low ($0.10$--$0.32$), at ceiling (ALOI, above $0.98$), or the set was
too small to separate the constructions (Yale, all within $0.01$).

Table~\ref{tab:main} separates two questions: does the symmetric basis 
reach a higher rate, and how much storage does it need to match \textsf{plain}.
The control is reported the same way.

Two effects combine in the saving.  Per retained vector the symmetric basis
stores half as many numbers, because it keeps only the top half of $\tens{U}$
(Theorem~\ref{thm:symsvd}), while the random-pairing control shares this factor.  On
top of that, when the mirror pairing concentrates the useful energy
in fewer vectors, \textsf{new} reaches the bar with fewer of them, and the
saving grows.  On LFW, for example,
\textsf{plain} first reaches the bar with $4595$ stored numbers
($\gamma=0.85$) and \textsf{new} with $678$ ($\gamma=0.8$), a $6.8\times$
saving. Overall we observe $5$--$11\times$ on the face datasets (LFW $6.8\times$, CelebA
$5.3\times$, Olivetti $7.3\times$, AFHQ $11.3\times$), $2$--$5\times$ on
leaves and butterfly wings, while the random pairing needs roughly
$1.5$--$4\times$ more storage than the mirror pairing on the same datasets,
or never reaches the bar (LeafSnap, where neither construction has a rate
gain, is the one exception).  On the recognition-rate axis the symmetric basis is clearly above
both \textsf{plain} and \textsf{rand} on Extended Yale~B and AFHQ
($+0.073$ to $+0.102$ over \textsf{plain},
$+0.035$ to $+0.069$ over \textsf{rand} on average over the grid), while on Olivetti,
LFW, CelebA, the leaves, and the butterfly wings
it matches \textsf{plain}'s rate at less storage without clearly separating
from the random-pairing control. On Extended Yale~B, where illumination varies
but the face itself is symmetric, \textsf{plain} never reaches the symmetric
basis's recognition rate at any storage: on the $32$ mildest illuminations
its best is $0.841$, and it needs $75{,}725$ stored numbers to come within
$0.01$ of it, while \textsf{new} passes $0.841$ outright with $710$ and
reaches $0.943$. For reference, the t-SVD recognition rates reported on this
database under very similar but not identical protocols are $0.79$ on average
(first $20$ illuminations, $15$ training images per subject, $\gamma=0.9$
\cite[Table~1]{HaoKBH2013}) and about $0.95$ (first $30$ illuminations,
$10$-fold cross-validation, fixed truncation \cite[Tables~3--4]{ZhangSKA2018}).
With all $64$ illuminations, every rate drops (\textsf{plain} $0.707$,
\textsf{new} $0.801$), but the ordering stays the same. 
The tabulated $106.6\times$ and $40.9\times$ are therefore comparisons against \textsf{plain}'s ceiling,
and the informative view is Figure~\ref{fig:curves}.  Observe that on this database the random-pairing
control itself exceeds \textsf{plain}'s best rate ($0.919$ against $0.841$).
One explanation is that any fold averages away part of the illumination variation,
so even a wrong pairing helps. The mirror pairing does so without blurring the face,
and gains even more. When the symmetric basis cannot reach the bar at any
$\gamma$ (MUCT), there is no matched point, shown as ``---''.  A large
$\times$less does not by itself mean that reflection helped: on a near-chance
set \textsf{plain} needs many components just to reach its own low bar, which
inflates the ratio, while the \textsf{new}$-$\textsf{rand} column is the check.

\begin{table}[!htbp]
\centering
\resizebox{\textwidth}{!}{
\begin{tabular}{lccccccccccc}
\toprule
 & \multicolumn{3}{c}{best recognition rate} & \multicolumn{2}{c}{gap} & \multicolumn{3}{c}{storage to reach the bar} & \multicolumn{2}{c}{$\times$less} \\
\cmidrule(lr){2-4}\cmidrule(lr){5-6}\cmidrule(lr){7-9}\cmidrule(lr){10-11}
dataset & \textsf{plain} & \textsf{new} & \textsf{rand} & \textsf{new}$-$\textsf{plain} & \textsf{new}$-$\textsf{rand}
  & \textsf{plain} & \textsf{new} & \textsf{rand} & \textsf{new} & \textsf{rand} \\
\midrule
\multicolumn{11}{l}{Win: \textsf{new} above both \textsf{plain} and \textsf{rand} (rule in caption)}\\
\quad Extended Yale~B & 0.841 & 0.943 & 0.919 & $+0.102$ & $+0.023$ & 75725 & 710 & 1843 & $106.6\times^\dagger$ & $41.1\times^\dagger$\\
\quad Extended Yale~B (all illuminations) & 0.707 & 0.801 & 0.756 & $+0.094$ & $+0.045$ & 57779 & 1414 & 4915 & $40.9\times^\dagger$ & $11.8\times^\dagger$\\
\quad AFHQ & 0.519 & 0.592 & 0.554 & $+0.073$ & $+0.038$ & 3264 & 288 & 838 & $11.3\times$ & $3.9\times$\\
\midrule
\multicolumn{11}{l}{Storage saving at comparable rate}\\
\quad Olivetti & 0.935 & 0.957 & 0.948 & $+0.022$ & $+0.008$ & 4275 & 582 & 1235 & $7.3\times$ & $3.5\times$\\
\quad LFW & 0.302 & 0.357 & 0.330 & $+0.055$ & $+0.027$ & 4595 & 678 & 1376 & $6.8\times$ & $3.3\times$\\
\quad CelebA & 0.338 & 0.371 & 0.338 & $+0.032$ & $+0.032$ & 5798 & 1101 & 1702 & $5.3\times$ & $3.4\times$\\
\quad Swedish leaves & 0.890 & 0.914 & 0.886 & $+0.024$ & $+0.028$ & 2867 & 877 & 3565 & $3.3\times$ & $0.8\times$\\
\quad Leeds butterflies & 0.859 & 0.859 & 0.839 & $+0.000$ & $+0.020$ & 8627 & 1658 & --- & $5.2\times$ & ---\\
\quad LeafSnap & 0.921 & 0.918 & 0.917 & $-0.003$ & $+0.001$ & 3136 & 1562 & 966 & $2.0\times$ & $3.2\times$\\
\midrule
\multicolumn{11}{l}{No effect (at ceiling)}\\
\quad COIL-100 & 0.984 & 0.991 & 0.984 & $+0.007$ & $+0.007$ & 973 & 448 & 480 & $2.2\times$ & $2.0\times$\\
\midrule
\multicolumn{11}{l}{Loss}\\
\quad MUCT & 0.819 & 0.787 & 0.776 & $-0.032$ & $+0.011$ & 7526 & --- & --- & --- & ---\\
\bottomrule
\end{tabular}}
\caption{Recognition and storage for the three constructions on ten datasets,
plus Extended Yale~B with all $64$ illuminations (protocol
  of Section~\ref{sec:protocol}: orthonormal DCT, energy truncation of
  Algorithm~\ref{alg:symsvd}, Hao et al.'s recognition procedure, $70/30$
  split, classes as in Section~\ref{sec:setup}, $5$-seed mean).  ``Best recognition rate'' is
  each construction's best over the $\gamma$ grid. The gaps are \textsf{new}'s best
  minus that of the others.  ``Storage to reach the bar'' is the smallest number of
  stored transform-domain basis entries ($m\sum_i\rho_i$ for \textsf{plain},
  $h\sum_i\rho_i$ for \textsf{new} and \textsf{rand}) at which the construction's rate
  reaches \textsf{plain}'s best minus $0.01$, $\times$less is \textsf{plain}'s
  storage divided by that construction's, and ``---'' means the bar is never reached.
  Groups: a \emph{win} means \textsf{new}'s best rate is at least $0.03$
  above \textsf{plain}'s and \textsf{new} is above \textsf{rand} by at
  least $0.02$ on average over the $\gamma$ grid (the margin of
  Section~\ref{sec:metrics}), \emph{storage saving} means
  \textsf{new} reaches \textsf{plain}'s rate at less storage without a
  clear rate gain. Because storage is read off the $\gamma$ grid, the
  $\times$less ratios are coarse, while the rates are not. 
  \\
  $^\dagger$On the Extended
  Yale~B rows \textsf{plain} never reaches the symmetric basis's rate, so the
  ratio compares against \textsf{plain}'s ceiling.}
\label{tab:main}
\end{table}

\begin{figure}[!htbp]
\centering
\includegraphics[width=0.9\textwidth]{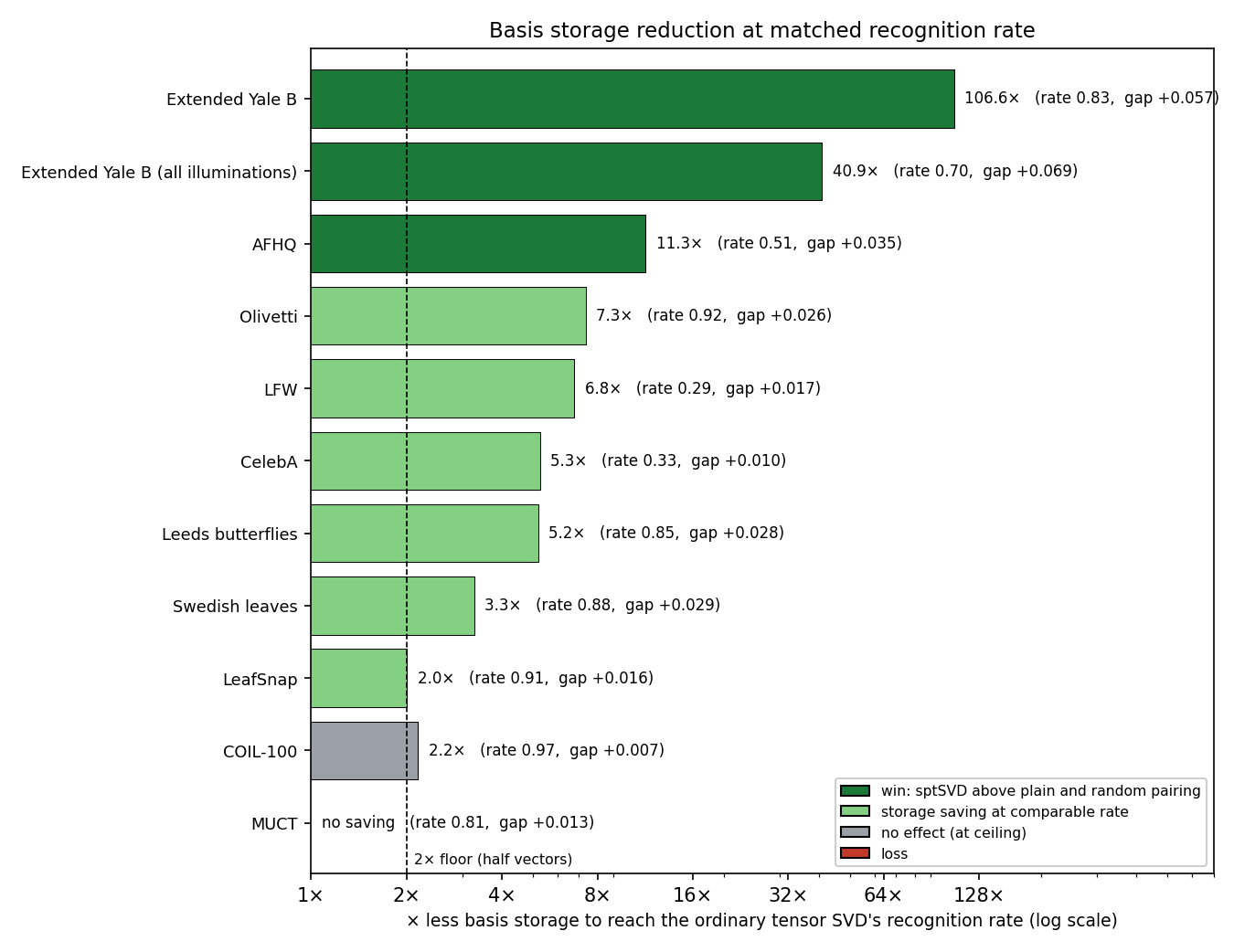}
\caption{Basis storage reduction ($\times$less) to match the ordinary tensor SVD's recognition rate, the datasets of Table~\ref{tab:main}, colored by group (dark green: win, light green: storage saving at comparable rate, grey: no effect, red: loss).  The horizontal axis is logarithmic, the dashed
line marks the automatic $2\times$ from storing half of each vector.  The
Extended Yale~B ratios compare against \textsf{plain}'s ceiling.}
\label{fig:main}
\end{figure}

\subsection{Recognition across the storage range}\label{sec:curves}
Figure~\ref{fig:curves} plots recognition rate against stored basis size
(all faces, transform domain) for the datasets of Table~\ref{tab:main}
and all three constructions, as $\gamma$ sweeps the grid.
The plotted rates range from below $0.1$ (at the smallest $\gamma$) to
$0.99$, so a common vertical axis would flatten most panels.
Please note that each panel has its own vertical scale, so gaps should
not be compared by eye across panels.
Equal $x$ means equal stored numbers, so the vertical gap between \textsf{new}
and \textsf{plain} is the gain at equal storage
and the horizontal gap at a given rate is the storage saving. 
On every face dataset the symmetric basis is above \textsf{plain} at every
storage level beyond the first few hundred numbers, and above the
random-pairing control on average (they cross at a few small-storage points,
mainly on LFW and CelebA). The gap is 
largest under varying illuminations (Extended Yale~B, $+0.3$ at equal storage)
and smallest where the rate is near its ceiling.
Leaves and butterfly wings show the same
ordering as the faces with smaller margins (and isolated crossings with \textsf{rand}).
At ceiling (COIL-100) the three curves converge: \textsf{new}
reaches the $0.98$ plateau with fewer stored numbers, but once every
construction is there no basis choice can matter.  MUCT is the one dataset
where \textsf{plain} overtakes the symmetric constructions at large storage.
There the antisymmetric part that the symmetric basis discards helps distinguish faces,
so the symmetric constructions level off below \textsf{plain}'s best rate
(Section~\ref{sec:when}).

\begin{figure}[!htbp]
\centering
\includegraphics[width=0.95\textwidth]{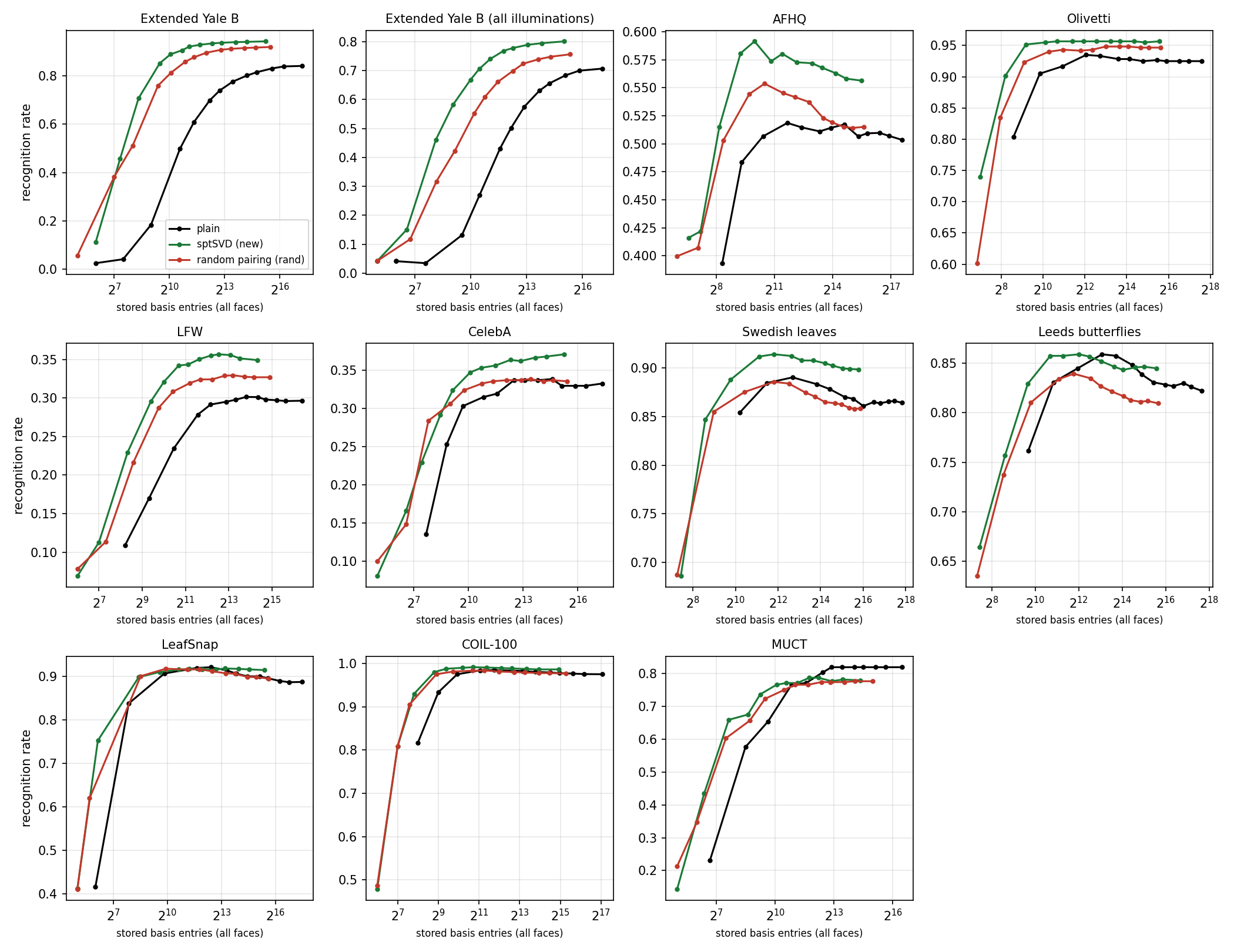}
\caption{Recognition rate against stored basis size (log scale) for 
    the datasets of Table~\ref{tab:main}, as $\gamma$ sweeps the grid.
    Same protocol as Table~\ref{tab:main}, $5$-seed mean. 
    Vertical scales differ between panels.}
\label{fig:curves}
\end{figure}

\subsection{Reconstruction}\label{sec:recon}
On perfectly symmetric data the symmetry-preserving truncation has the same
error as the ordinary $\starM$-SVD at every $k$ at half the storage. 
On real images the method reconstructs the symmetric part
of each image only, so its reconstruction error is bounded below by the energy
of the discarded antisymmetric part, and at equal storage the unconstrained
basis reconstructs the raw pixels more accurately. The benefit is thus in recognition,
not reconstruction. On aligned symmetric objects the discarded part is mostly
illumination and pose, rather than identity. Figure~\ref{fig:recon-gallery}
shows what the symmetric reconstruction preserves.

\begin{figure}[!htbp]
\centering
\includegraphics[width=0.55\textwidth]{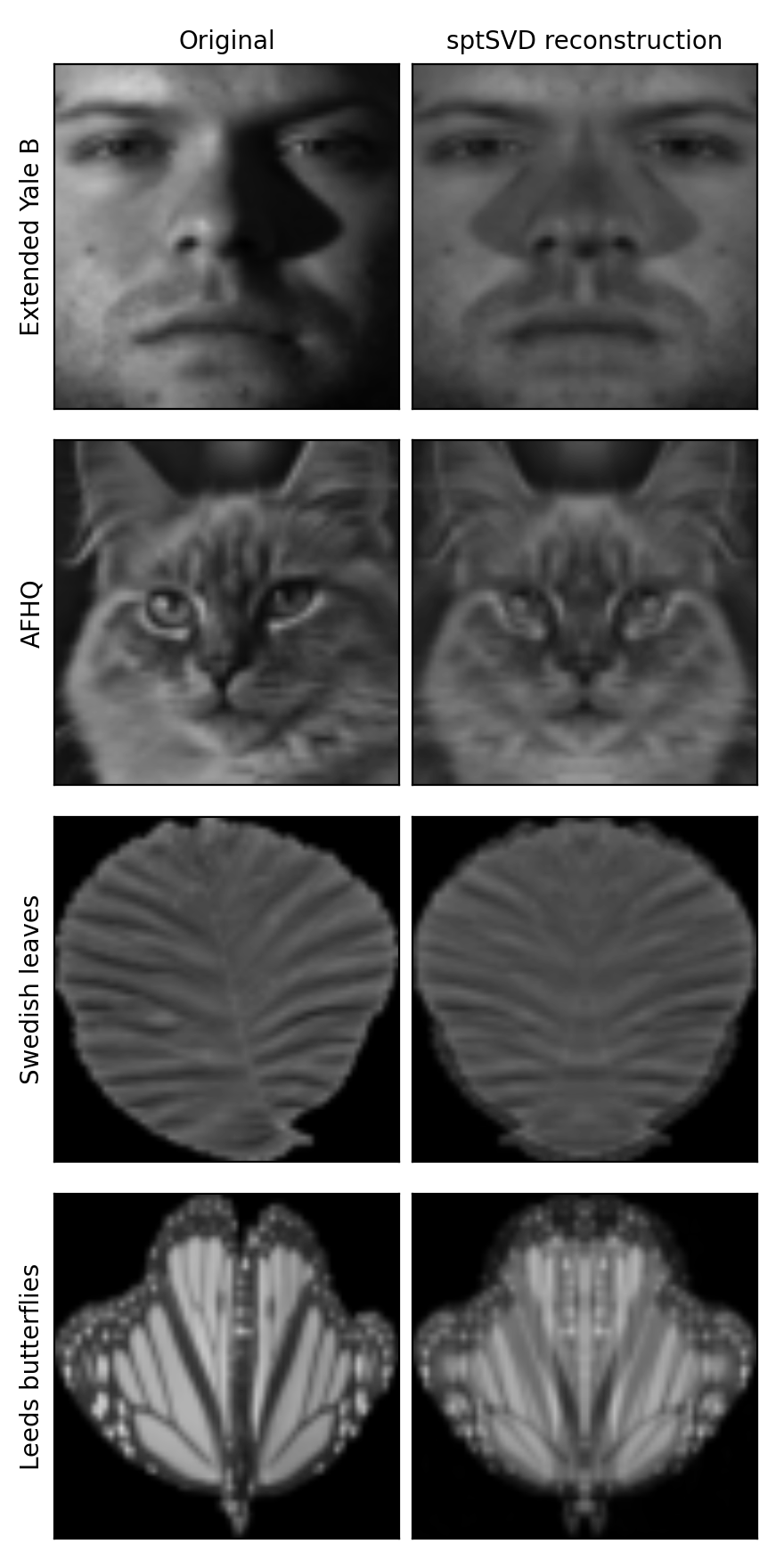}
\caption{sptSVD reconstructions at $k=h$ (the largest rank the symmetric method can use) for one test image from each of four datasets spanning distinct domains.}
\label{fig:recon-gallery}
\end{figure}

\subsection{When the benefit appears}\label{sec:when}
The reflection benefit requires aligned, bilaterally symmetric data whose
identity resides in the symmetric component.  Every dataset in the top group of
Table~\ref{tab:main} is of that kind: frontal faces, and pre-aligned leaves and
butterfly wings. 

MUCT is the only face dataset where the symmetric basis never
reaches the ordinary rate, and it does not beat the random
pairing there either, meaning the fold loses information the
reflection does not recover. MUCT was captured with five cameras
(one frontal, two three-quarter, two elevated). Re-running on each
camera group separately gives the same loss on frontal and elevated
views (\textsf{plain} $0.628$, \textsf{new} $0.575$, \textsf{rand} $0.617$)
and a tie on three-quarter views ($0.364$, $0.361$, $0.368$), and re-centering
each image on its symmetry axis does not change this. So the loss is not from
pose or crop, and clearly on these faces the antisymmetric part matters.

Where recognition is low (beetles, cars, BIOSCAN: best rates $0.10$--$0.23$ against chance of
$0.07$--$0.08$), aligning the objects does not create a
\textsf{new}$-$\textsf{rand} margin. There is little class signal in the
symmetric part of the image for the basis to keep, so the symmetric
constructions gain nothing beyond storing half of each vector.

\subsection{Relationship to prior work}
Relative to the matrix symmetry-preserving SVD of Shah and Sorensen
\cite{ShahSorensen2006}, what is new is the extension to third-order
tensors under an orthogonal transform and the pooled truncation. Relative to the
$\starM$-SVD \cite{KernfeldKilmerAeron2015,KHAN2021}, it is the
symmetry-preserving specialization and its half-storage. The symmetry gives the
half storage and the transform gives the multiway compression, and we combine them.

Against the non-symmetric $\starM$-SVD, the symmetric decomposition matches the \\
recognition rate at reduced basis storage.
Randomized t-SVD \cite{ZhangSKA2018}
is independent of symmetry and could be applied to the folded slices, though at
our image sizes the exact SVDs are not a bottleneck.

\section{Conclusions and Future Work}\label{sec:conclusion}
We defined $\starM$-symmetric tensors as those whose transform-domain
frontal slices are reflectively symmetric, and established the symmetry-preserving
$\starM$-SVD $\tens{A}=\tens{U}\starM\tens{S}\starM\tens{V}\trans$ with an exactly
symmetric, half-stored left basis. The construction extends the matrix
symmetry-preserving SVD of Shah and Sorensen to the $\starM$-algebra. In recognition experiments it matches the ordinary tensor SVD's rate at $2$--$11\times$ less basis storage, except on MUCT, and on faces under varying illumination exceeds its best rate at any storage.

Several directions remain.  The clearest one is to understand what happens when the
data is only approximately symmetric: our error identity assumes exact symmetry in
each fold, and it would be useful to know how it degrades without that assumption.
A second is the choice of transform $\Mmat$: a data-adapted transform,
learned as in \cite{NewmanKeegan2025}, might compress better while keeping the
orthogonality the method relies on.  Finally, the same
fold-and-halve idea should extend to data with more than one symmetry axis (for
instance an object symmetric both left--right and top--bottom), which would cut
storage further. However, imposing more symmetry throws away more of the data,
so past some point the extra storage saving costs recognition rate.

\medskip
\section*{Acknowledgements}
MK’s work on this project was
partially supported by NSF DMS-2410698.

\section*{Data sources}
The recognition datasets of Section~\ref{sec:experiments} are, with their sources:
{\small
\begin{itemize}\setlength{\itemsep}{1pt}
\item \textbf{Extended Yale~B} (cropped): A.~S.~Georghiades, P.~N.~Belhumeur,
  D.~J.~Kriegman, \emph{From few to many: illumination cone models for face
  recognition under variable lighting and pose}, IEEE TPAMI 23(6):643--660, 2001;
  K.-C.~Lee, J.~Ho, D.~J.~Kriegman, \emph{Acquiring linear subspaces for face
  recognition under variable lighting}, IEEE TPAMI 27(5):684--698, 2005.
  
  \url{https://vision.ucsd.edu/~leekc/ExtYaleDatabase/ExtYaleB.html}
\item \textbf{Yale} (original): P.~N.~Belhumeur, J.~P.~Hespanha, D.~J.~Kriegman,
  \emph{Eigenfaces vs.\ Fisherfaces}, IEEE TPAMI 19(7):711--720, 1997.
\item \textbf{LFW}: G.~B.~Huang, M.~Ramesh, T.~Berg, E.~Learned-Miller,
  \emph{Labeled Faces in the Wild}, Univ.\ Massachusetts Amherst TR~07-49, 2007.
  \url{http://vis-www.cs.umass.edu/lfw/}
\item \textbf{CelebA}: Z.~Liu, P.~Luo, X.~Wang, X.~Tang, \emph{Deep learning face
  attributes in the wild}, ICCV 2015. \url{https://mmlab.ie.cuhk.edu.hk/projects/CelebA.html}
\item \textbf{AFHQ}: Y.~Choi, Y.~Uh, J.~Yoo, J.-W.~Ha, \emph{StarGAN v2}, CVPR
  2020. \url{https://github.com/clovaai/stargan-v2}
\item \textbf{Olivetti/ORL}: F.~S.~Samaria, A.~C.~Harter, \emph{Parameterisation of
  a stochastic model for human face identification}, IEEE WACV 1994.
\item \textbf{MUCT}: S.~Milborrow, J.~Morkel, F.~Nicolls, \emph{The MUCT
  landmarked face database}, PRASA 2010. \url{http://www.milbo.org/muct/}
\item \textbf{Swedish leaves}: O.~J.~O.~S\"oderkvist, \emph{Computer vision
  classification of leaves from Swedish trees}, MSc thesis, Link\"oping Univ.,
  2001. \url{https://www.cvl.isy.liu.se/research/datasets/swedish-leaf/}
\item \textbf{LeafSnap}: N.~Kumar et al., \emph{Leafsnap: a computer vision system
  for automatic plant species identification}, ECCV 2012.
  \url{https://leafsnap.com/dataset/}
\item \textbf{Leeds Butterfly}: J.~Wang, K.~Markert, M.~Everingham, \emph{Learning
  models for object recognition from natural language descriptions}, BMVC 2009.
  
\item \textbf{COIL-100}: S.~A.~Nene, S.~K.~Nayar, H.~Murase, \emph{Columbia Object
  Image Library (COIL-100)}, Columbia Univ.\ TR CUCS-006-96, 1996.
  \url{https://www.cs.columbia.edu/CAVE/software/softlib/coil-100.php}
\item \textbf{ALOI}: J.-M.~Geusebroek, G.~J.~Burghouts, A.~W.~M.~Smeulders,
  \emph{The Amsterdam Library of Object Images}, IJCV 61(1):103--112, 2005.
  \url{https://aloi.science.uva.nl/}
\item \textbf{Stanford Cars}: J.~Krause, M.~Stark, J.~Deng, L.~Fei-Fei, \emph{3D
  object representations for fine-grained categorization}, ICCV Workshops (3dRR)
  2013.
\item \textbf{BIOSCAN}: Z.~Gharaee et al., \emph{BIOSCAN-5M: a multimodal dataset
  for insect biodiversity}, NeurIPS 2024 (30k subset).
  \url{https://huggingface.co/datasets/Voxel51/BIOSCAN-30k}
\item \textbf{Beetles} (NEON ground beetles, Carabidae): \emph{2018 NEON
  ethanol-preserved ground beetles}, Imageomics.
  \\\url{https://huggingface.co/datasets/imageomics/2018-NEON-beetles}
\end{itemize}}

\bibliographystyle{siamplain}
\bibliography{refs}

\end{document}